\documentclass[11pt,letterpaper]{amsart}

\usepackage{amssymb,mathtools,amsthm}
\usepackage[a4paper,left=3cm,right=3cm,top=3cm,bottom=3cm]{geometry}
\usepackage[T1]{fontenc}
\usepackage{lmodern}
\usepackage[utf8]{inputenc}
\usepackage{microtype}
\usepackage{enumitem}
\usepackage{caption}
\usepackage{dsfont}
\usepackage{subcaption}
\usepackage{bbm, dsfont}
\setenumerate[1]{label={\arabic*.},ref={\arabic*}}

\mathtoolsset{showonlyrefs}

\newcounter{n}
\newcounter{bigresults}
\numberwithin{n}{section}
\theoremstyle{plain}

\newtheorem{lemma}[n]{Lemma}
\newtheorem{theorem}[bigresults]{Theorem}
\newtheorem*{theorem*}{Theorem}

\newtheorem{proposition}[n]{Proposition}
\newtheorem{corollary}[n]{Corollary}
\theoremstyle{definition}
\newtheorem{definition}[n]{Definition}
\newtheorem*{definition*}{Definition}

\newtheorem*{remark*}{Remark}

\newtheorem*{claim*}{Claim}
\newtheorem*{assertion*}{Assertion}
\newtheorem*{proposition*}{Proposition}

\theoremstyle{remark}

\usepackage{tikz}
\usetikzlibrary{arrows.meta}
\usepackage{subcaption}

\usepackage{hyperref}
\definecolor{colorlinks}{RGB}{0, 24, 168}
\definecolor{colorcites}{RGB}{124, 10, 2}
\hypersetup{
	colorlinks=true,
	linkcolor=colorlinks,
	citecolor=colorcites,
	urlcolor=colorlinks,
	pdfborder={0 0 0}
}

\renewcommand\varphi\phi
\renewcommand\epsilon\varepsilon
\newcommand\eps\varepsilon
\renewcommand\setminus\smallsetminus

\renewcommand\O[1]{\(\mathrm{O}(#1)\)}

\newcommand\T{\mathbb T}

\newcommand\E{\mathbb E}

\newcommand\bbS{\mathbb S}
\newcommand\C{\mathbb C}
\renewcommand\P{\mathbb P}
\newcommand\Z{\mathbb Z}
\renewcommand\H{\mathbb H}

\renewcommand\L{\mathbb L}
\newcommand\N{\mathbb N}

\newcommand\calC{\mathcal C}

\newcommand\calF{\mathcal F}
\newcommand\calG{\mathcal G}

\newcommand\calS{\mathcal S}

\newcommand\muLoop{{\sf Loop}}

\newcommand\hexlattice{{\mathbb H}}
\newcommand\spaceLoop[1]{\mathfrak{S}_{\muLoop}(#1)}

\newcommand\face{F}
\newcommand\faces[1]{\face(#1)}
\newcommand\domain\Omega
\newcommand\vertices[1]{{V(#1)}}
\newcommand\edges[1]{{E(#1)}}
\renewcommand\varnothing\emptyset

\newcommand{\xlra}{\xleftrightarrow} %for path connections
\newcommand\upvert{\mathrm{Y}} %set of vertices corresponding to upward-oriented triangles
\newcommand{\1}{\mathbbm{1}}

\newcommand\ballloop{B}

\renewcommand\subset\subseteq

\author{Alexander Glazman}
\address{Universität Innsbruck, Innsbruck, Austria}
\email{alexander.glazman@uibk.ac.at}

\author{Matan Harel}
\address{Northeastern University, Boston, USA}
\email{m.harel@northeastern.edu}

\author{Nathan Zelesko}
\address{Northeastern University, Boston, USA}
\email{zelesko.n@northeastern.edu}

\title[Planar percolation and the loop \O{n} model]{Planar percolation and the loop \O{n} model}

\date{\today}

\makeatletter
\@namedef{subjclassname@2020}{\textup{2020} Mathematics Subject Classification}
\makeatother
\begin{document}

\begin{abstract}
	We show that a large class of site percolation processes on any planar graph contains either zero or infinitely many infinite connected components. The assumptions that we require are: tail triviality, positive association (FKG), and that the set of open vertices is stochastically dominated by the set of closed ones.
	This covers the case of Bernoulli site percolation at parameter~$p\leq 1/2$ and resolves Conjecture~$8$ from the work of Benjamini and Schramm from~1996.
	Our result also implies that~$p_c\geq 1/2$ for any invariantly amenable unimodular random rooted planar graph.
	
	Furthermore, we apply our statement to the loop~\O{n} model on the hexagonal lattice and confirm a part of the phase diagram conjectured by Nienhuis in 1982: the existence of infinitely many loops around every face whenever~$n\in [1,2]$ and~$x\in [1/\sqrt{2},1]$.
	The point~$n=2,x=1/\sqrt{2}$ is conjectured to be critical.
	This is the first instance that this behavior has been proven in such a large region of parameters.
	In a big portion of this region, the loop~\O{n} model has no known FKG representation. We apply our percolation result to quenched distributions that can be described as divide and color models.
\end{abstract}

\maketitle

\section{Introduction}
\label{sec:intro}

Modern statistical physics has its origin in the study of the Lenz-Ising model and its phase transition. This model of ferromagnetic interactions was first defined more than a century ago. It was later generalized to a large class of spin lattice models meant as a discrete approximation to Euclidean geometry. A few decades later, Broadbent-Hammersley~\cite{BroHam57} introduced Bernoulli percolation. 
This is a model of random subgraphs of a given (crystallographic) lattice undergoing a phase transition in terms of the existence of an infinite connected component.
It was then discovered that percolation models provide a graphical representation for correlations in Ising-like spin models.
Since then, percolation models have been at the center of the mathematical approach to phase transitions; see~\cite{Gri99a} for a classical manuscript and~\cite{Man25} for a recent survey.

Special attention has been lavished upon two-dimensional lattice models, where the scaling limits of lattice approximations are expected to be described by a conformal field theory.
These emergent symmetries have lead to the wide-open conjecture that certain interfaces converge to the Schramm--Loenwer Evolution (SLE). 
This has been proved in only a handful of models, including critical Ising model on isoradial graphs~\cite{Smi10,CheSmi11} and critical site percolation on the triangular lattice~\cite{Smi01, Smi01a}.

Much of the study of percolation is focused on vertex (quasi-)transitive graphs, where ergodic-theoretic arguments can be used to restrict the possible behavior of Bernoulli percolation.
Indeed, denote by~$N_\infty$ the number of infinite connected components.
Then, it is straightforward to show that, on any connected transitive graph, $N_\infty$ is almost surely constant, and can only be zero, one, or infinity. If the graph is amenable, a powerful argument of Burton and Keane~\cite{BurKea89} rules out the case~$N_\infty=\infty$.
Also, if $N_\infty=1$ for Bernoulli percolation of some parameter $p$, then the same holds for any~$q>p$~\cite{HagPer99}.
Both statements are false if one considers general planar graphs; see Fig.~\ref{fig:countre-examples} for examples where $N_\infty$ has a non-trivial distribution and the property $N_\infty=1$ is non-monotone in~$p$.

\begin{figure}[t]
\centering \begin{subfigure}[]{0.4 \textwidth}
\begin{center}\includegraphics[scale=0.8,page=1]{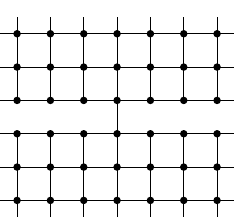} \end{center}
\label{fig:subfig1a} 
\end{subfigure}  \begin{subfigure}[]{0.4 \textwidth}
\begin{center} \includegraphics[scale=0.8,page = 2]{Nccexamples.pdf} \end{center}
\label{fig:subfig1b} \end{subfigure} \caption{\textsc{Left}: two copies of $\mathbb{Z} \times \mathbb{Z}^+$ connected by one edge; at any $p > p_c(\mathbb{Z}^2)$, the probability that the number of infinite components is one or two is positive. \textsc{Right}: a half plane connected to a graph of branching number $\sqrt{2}$. For $p_c(\mathbb{Z}^2) < p \leq 1/\sqrt{2}$, Bernoulli percolation contains  a unique infinite component, almost surely; for $1/\sqrt{2}< p < 1$, there are infinitely many connected components in Bernoulli percolation.}
\label{fig:countre-examples}
\end{figure}

In this work we study general locally finite planar graphs and rule out the possibility that $1 \leq N_\infty < \infty$ for  Bernoulli site percolation whenever~$p \leq 1/2$. In fact, we prove this result for a more general class of site percolation models (Theorem~\ref{thm:perco}).
We emphasize that Theorem~\ref{thm:perco} does not require any symmetry, and also holds true when planar embeddings of~$G$ have any number of accumulation points. 
The case of Bernoulli percolation at $p=1/2$ settles Conjecture~$8$ in the seminal work of Benjamini and Schramm~\cite{BenSch96}. When~$G$ is quasi-transitive, this conjecture has been established recently by Grimmett and Li~\cite{GriLi25}.

In instances where one can rule out the scenario where $N_{\infty} = \infty$ {\em a priori}, Theorem~\ref{thm:perco} implies the absence of infinite connected components and provides the optimal lower bound on the percolation threshold~$p_c$.  We show that~$p_c\geq 1/2$ for any invariantly amenable unimodular random rooted planar graph (Corollary~\ref{cor:p-c}).
This improves on a uniform lower bound proved by Peled~\cite{Pel20} and can be thought of as a generalization of the no-coexistence theorems of Zhang (unpublished, described in~\cite[Lemma~9.12]{Gri99a}) and Sheffield~\cite{She05} (see also~\cite{DumRaoTas19}).
Such results have been very useful in proving delocalization of height functions~\cite{ChaPelSheTas21,Lam19,GlaMan21} which allowed for new proofs of the continuity of the phase transition~\cite{GlaLam25} and the existence of a Berezinskii–Kosterlitz–Thouless (BKT) phase~\cite{LisEng23,AizHarPelSha21}; see~\cite{FroSpe81,DumSidTas17} for the original proofs. 
Since we do not rely on the symmetries of the graphs, our method opens a path to extend all these results to more general settings.

Corollary~\ref{cor:divide-color} extends the statement~$p_c\geq 1/2$ of Corollary~\ref{cor:p-c} to divide and color models, a general class which includes the Ising and fuzzy Potts models as particular cases. As we show below, the loop \O{n} model also fits in this class.
This model is supported on collections of non-intersecting cycles (loops) on the hexagonal lattice and has parameters~$n>0$ (loop-weight) and~$x>0$ (edge-weight). 
It was introduced in 1981~\cite{DomMukNie81} and is difficult to study due to the absence of monotonicity. This model is conjectured to undergo a phase transition in terms of loop lengths for all~$n\in (0,2]$~\cite{Nie82}: macroscopic loops when~$x\geq x_c(n)$ versus exponential decay when~$x< x_c(n)$.

We use Theorem~\ref{thm:perco} via Corollary~\ref{cor:divide-color} to establish a large portion of the conjectured phase diagram of the loop \O{n} model (Theorem~\ref{thm:LoopOn}). 
Specifically, we prove that each face is surrounded by infinitely many loops when~$n \in [1,2], x \in [1/\sqrt{2},1]$.
This covers nearly all previously known results of this sort~\cite{DumGlaPel21,GlaMan21,CraGlaHarPel25,GlaLam25}.
Together with the dichotomy shown in~\cite{DumGlaPel21}, our result implies the macroscopic behavior when~$n \in [1,2], x \in [1/\sqrt{2},1/\sqrt{n}]$.

The main idea is to use Edwards--Sokal-type graphical representations and apply Theorem~\ref{thm:perco} twice to the respective conditional distributions. These quenched measures will be positively associated as independent percolations  -- while the positive association fails for the annealed measures. We hope that our approach of applying positive association for conditional measures will be useful for the study of other models.

\subsection{Site percolation models on planar graphs}

Let~$G=(V,E)$ be a planar graph.
It is called locally finite if each vertex has a finite degree.
We consider a product $\sigma$-algebra~$\calF$ on~$\{0,1\}^V$. 
We say that~$A\in\calF$ is a tail event if its occurrence cannot be altered by changing the state of finitely many vertices.
We introduce a pointwise partial order on~$\{0,1\}^V$: for~$\sigma,\sigma'\in \{0,1\}^V$, we say that~$\sigma\leq\sigma'$ if~$\sigma_v\leq \sigma'_v$ for any~$v\in V$.
An event is called increasing if its indicator function is non-decreasing with respect to this partial order.

A site percolation process is a random variable~$\sigma$ with values in~$\{0,1\}^V$; a vertex~$v\in V$ is called open if~$\sigma_v=1$ and closed otherwise.
The most prominent example is Bernoulli site percolation at some parameter~$p\in [0,1]$, where each vertex is assigned~$1$ with probability~$p$, independently of all others.
We often identify~$\sigma$ with the subgraph of~$G$ induced by the set of open vertices.
We say that~$\sigma$ is
\begin{itemize}
	\item {\em tail trivial} if, for any tail event~$A$,
	\[
		\P(\sigma\in A)\in \{0,1\};
	\]
	\item {\em positively associated} if, for any two increasing events~$A$ and~$B$,
	\[
		\P(\sigma\in A\cap B)\geq \P(\sigma\in A) \cdot \P(\sigma\in B);
	\]
	\item {\em stochastically dominated} by~$1-\sigma$ if, for any increasing event~$A$,
	\[
		\P(\sigma\in A) \leq \P((1-\sigma)\in A).
	\]
\end{itemize}
In the last item, the random variable~$1-\sigma$ is defined by~$(1-\sigma)_v:=1-\sigma_v$ for each~$v\in V$.
By the Kolmogorov zero-one law and the Harris inequality~\cite{Har60}, Bernoulli site percolation satisfies the first two properties; whenever~$p\leq 1/2$, it also satisfies the third.

\begin{theorem}\label{thm:perco}
	Let~$G=(V,E)$ be an infinite locally finite planar graph.
	Consider a random variable~$\sigma\colon V\to \{0,1\}$ that is tail trivial, positively associated and is stochastically dominated by $1-\sigma$.
	Then, $\{\sigma=1\}$ contains either no infinite connected component a.s. or infinitely many of them a.s.
\end{theorem}

The study of percolation on general planar graphs was initiated by Benjamini and Schramm~\cite{BenSch96}.
That celebrated work asks which properties of Bernoulli site percolation are particular to two-dimensional lattices and which ones hold in general.
Their Conjecture~$8$ states the claim of Theorem~\ref{thm:perco} for Bernoulli site percolation at~$p=1/2$.
As mentioned above, this model satisfies the assumptions of Theorem~\ref{thm:perco} and, thus, we settle the conjecture completely.
The statement at~$p=1/2$ for quasi-transitive planar graphs is the content of~\cite[Theorem~1.9]{GriLi25}.

Theorem~\ref{thm:perco} is also linked to~\cite[Conjecture~$7$]{BenSch96} that states that, for any planar graph of minimal degree at least~$7$, $p_c<1/2$ and~$p_u\geq 1-p_c$, where $p_c$ and~$p_u$ are defined by
\begin{align*}
	p_c(G) &:= \inf \{ p : \mathbb{P}[G_p \text{ contains an infinite connected component}] =1\},\\
	p_u(G) &:= \inf \{ p : \mathbb{P}[G_p \text{ contains a unique infinite connected component}] > 0\}.
\end{align*}
In the case where~$G$ can be embedded in the plane with no accumulation points, this conjecture has been confirmed by Haslegrave and Panagiotis~\cite{HasPan21} ($p_c<1/2$) and Li~\cite{Li23} ($p_u\geq 1-p_c$).
Theorem~\ref{thm:perco} implies that $p_u \geq 1/2$ for any infinite locally finite planar graph, and~$p_u=p_c\geq 1/2$ if~$G$ is planar, unimodular, and invariantly amenable (Corollary~\ref{cor:p-c}). It is classical~\cite{BenSch01a} that~$p_u \geq 1 - p_c$ if~$G$ is planar, unimodular, and invariantly nonamenable. Building on an initial preprint of the current work, Li~\cite{Li26} proved that $p_u \geq 1 - p_c$ for any planar graph with countably many accumulation points in its Freudenthal embedding. She also constructed a planar graph where $p_u < 1 - p_c$. This graph must have uncountably many accumulation points in all of its planar embeddings.

Finally, we want to emphasize two robust features of the arguments presented in this work. 
First, our methods are not restricted to Bernoulli site percolation. Instead, we rely on the FKG inequality and tail triviality. This extends the study of planar percolations to this more general class of processes, which may allow for additional applications.
Second, our methods do not use any special properties of a `nice' choice of embedding of the planar graph $G$, such as a circle packing. Indeed, unlike most earlier works on planar percolations, we do not require the existence of a proper planar embedding (i.e. without accumulation points). We hope that our approach can be used to remove this extraneous assumption in other cases as well.

\subsection{Corollaries for percolation on unimodular invariantly amenable graphs}
\label{sec:divide-color}

Classical results of Burton and Keane~\cite{BurKea89} and Zhang (see~\cite[Chapter~11]{Gri99a}) imply that~$p_c=p_u\geq 1/2$ on any (quasi-)transitive amenable graph.
Using Theorem~\ref{thm:perco}, this can be extended to the case of unimodular invariantly amenable graphs (Corollary~\ref{cor:p-c}).
Let~$\calG_\bullet$ (resp.~$\calG_{\bullet\bullet}$) be the set of all isomorphism classes of locally finite planar graphs together with a distinguished vertex (resp. two distinguished vertices) equipped with the local topology.
A rooted random planar graph~$(G,\rho)$ is a random variable with values in~$\calG_{\bullet}$.
We say that~$(G,\rho)$ is {\em unimodular} if, for any Borel function~$f\colon \calG_{\bullet\bullet}\to [0,\infty]$, we have
\begin{equation}\label{eq:MTP}
	\E \left[ \sum_{v\in V} f(G, \rho, v) \right] = \E \left[ \sum_{v\in V} f(G, v, \rho) \right]. \tag{MTP}
\end{equation}
A bond percolation on a graph~$G=(V,E)$ is a random variable~$\omega$ that takes values in~$\{0,1\}^E$.
We identify~$\omega$ with the spanning subgraph of~$G$ given by~$\{\omega = 1\}$; we say that it is {\em finitary} if all connected components are finite, almost surely, and {\em invariant} if $(G,\rho,\omega)$ is unimodular in a sense that~\eqref{eq:MTP} holds also when $f$ depends on both $(G,\rho)$ and~$\omega$ (see~\cite[Section~2]{AngHutNacRay18} for more details).
For~$v\in V$, define~$K_\omega(v)$ as the connected component of~$v$ in~$\omega$;
for~$U\subset V$, define~$\partial U$ as the set of vertices in~$U$ that are adjacent to~$V\setminus U$.
We say that~$(G,\rho)$ is (vertex) {\em invariantly amenable} if
\[
	\inf \left\{ \E\left[\frac{|\partial K_\omega(\rho)|}{|K_\omega(\rho)|}\right] \colon \omega \text{ a finitary invariant percolation on } G \right\} = 0.
\]
Our definition via inner vertex boundaries agrees with~\cite{AldLyo07} and is different from~\cite{AngHutNacRay18}, which uses edge boundaries.
Note that if a graph is edge invariantly amenable in the sense of~\cite{AngHutNacRay18}, then it is also vertex invariantly amenable.
Moreover, when the root has a finite expected degree, being unimodular and (edge) invariantly amenable is equivalent to being the Benjamini--Schramm (or local) limit of finite planar maps~\cite{LipTar80,AngHutNacRay18}.

\begin{corollary}\label{cor:p-c}
	Let~$(G,\rho)$ be a rooted infinite locally finite planar random graph.
	Assume that its distribution is unimodular and invariantly amenable.
	Then,  Bernoulli site percolation on~$(G,\rho)$ at parameter~$p$ contains no infinite connected component if~$p\leq 1/2$, and at most one infinite connected component if~$p> 1/2$, almost surely.
	In particular,
	\[
		p_c(G)=p_u(G)\geq 1/2 \text{ a.s.}
	\]
\end{corollary}

Previously, Peled~\cite{Pel20} showed that there exists $c >0$ such that $p_c(G) \geq c$ for any planar, locally finite graph $G$ that may be represented by a circle packing with countably many accumulation points. Thanks to the celebrated work of He--Schramm~\cite{he1995hyperbolic}, this class of graphs includes all recurrent simple planar triangulations with countably many ends.
If view of~\cite{AngHutNacRay18}, the result of~\cite{Pel20} also implies $p_c(G) \geq c$ for every unimodular (edge) invariantly amenable with finite expected degree, since any circle packing of such graph has at most one accumulation point. We do not know whether all such graph are (vertex) invariantly amenable, but, as mentioned above, the implication is true when the root has a finite expected degree.
Corollary~\ref{cor:p-c} thus gives that~$p_c(G)\geq 1/2$ if~$(G,\rho)$ is the Benjamini--Schramm limit of finite planar maps and~$\rho$ has a finite expected degree.
This bound is optimal since the site percolation on the triangular lattice has~$p_c=1/2$~\cite{Kes80}.

We further extend the statement to the framework of divide and color models (Corollary~\ref{cor:divide-color}).
For a graph~$G=(V,E)$, we define $\mathrm{ER}(V)$ as the set of all equivalence relations on~$V$. 
Given~$\mu$ a probability measure on~$\mathrm{ER}(V)$, a {\em divide and color} model of parameter $p\in [0,1]$ is a site percolation $\sigma$ on $G$ defined as follows. First, one samples a partition $P = \{P_i\}_{i=1}^\infty$ from $\mu$. 
Then, for each $i \in \mathbb{N}$, we assign all vertices of $P_i$ the value $\sigma = 1$ with probability $p$ and $\sigma = 0$ otherwise, independently of all other classes. 
The model was first introduced by Häggström~\cite{Hag01} in the case when~$\mu$ is Bernoulli bond percolation; the critical value function in this case was studied in~\cite{BalBefTas13,BalBefTas13a}.
The Gibbs properties of the measure when~$\mu$ is the random-cluster model was considered in~\cite{Bal10}.

A version with a general partition~$\mu$ and several colors is due to Steif and Tykesson~\cite{SteTyk19}. Here we focus on a two coloring, which recovers the Ising, fuzzy Ising, and Voter models as particular cases, and allows us to apply Theorem~\ref{thm:perco} to the loop \O{n} model (Theorem~\ref{thm:LoopOn}).
We call a random partition~$P$ of~$V$ {\em finitary} if its elements are finite, almost surely, and {\em invariant} if~$(G,\rho,P)$ is unimodular.

\begin{corollary}\label{cor:divide-color}
	Let~$(G,\rho)$ be a rooted infinite locally finite unimodular invariantly amenable planar random graph.
	Sample~$\sigma$ from a divide and color model with parameter $p \leq 1/2$ induced by $\mu$, a finitary invariant partition. Then, $\sigma$ does not contain any infinite components of open sites, almost surely. 
\end{corollary}

Note that Corollary~\ref{cor:p-c} is a particular case of Corollary~\ref{cor:divide-color} obtained by taking the trivial partition made up of singletons. Another important particular case of Corollary~\ref{cor:divide-color} is when~$G$ is a fixed transitive amenable graph (eg. triangular lattice) and~$\mu$ is a finitary partition invariant to a group of transformations that acts transitively on~$G$ (eg. shifts).
This readily implies Proposition~\ref{prop:no-perc-xi}, the key ingredient in our proof of Theorem~\ref{thm:LoopOn}.

Corollaries~\ref{cor:p-c} and~\ref{cor:divide-color} follow from Theorem~\ref{thm:perco} once one rules out infinitely many infinite components.
The latter statement was proven in a number of cases: 
the classical argument of Burton and Keane~\cite{BurKea89} treats shift invariant site percolation processes on~$\Z^d$ with the finite energy property. This argument was then extended in~\cite[Sections~$6,8$]{AldLyo07} based on~\cite{LyoSch99,BenLyoPer99,BenLyoPerSch01} (see also~\cite[Theorem~5.9]{AngHutNacRay18}) to insertion-tolerant invariant bond percolation processes on unimodular graphs.
In~\cite[Theorem~7.9]{LyoPer17}, insertion-tolerant invariant site and bond percolation processes on transitive graphs were treated. The arguments in these works use insertion-tolerance only to construct an infinite component with at least three ends, so the assumption can be weakened.
The combination of these arguments extends the statement to the setting of Corollaries~\ref{cor:p-c} and~\ref{cor:divide-color}; we provide short proofs for completeness.

\subsection{Loop~\O{n} model} \label{sec:loopon}

 Let $\hexlattice=(\vertices{\hexlattice},\edges{\hexlattice})$ denote the
hexagonal lattice whose faces $\faces{\hexlattice}$ are centered at $\{k+\ell
e^{i\pi/3}:k,\ell\in\Z\}\subset\mathbb C$. 
A \emph{loop configuration} on~$\H$ is a spanning subgraph of~$\H$ in which every vertex has degree~$0$ or~$2$.
Denote the set of all loop configurations on~$\H$ by~$\spaceLoop{\H}$.
Note that each connected component of a loop configuration is either a cycle (that we call a {\em loop}) or a bi-infinite path.
A \emph{domain} is a finite subgraph $\domain=(\vertices{\domain},\edges{\domain})\subset\hexlattice$ consisting precisely of the
sets of vertices and edges which are on or contained inside a cycle on
$\hexlattice$.  

For any domain $\domain$ and~${\omega'}\in \spaceLoop{\H}$, we define
\[
\spaceLoop{\H;\domain;{\omega'}}
:=
\{
	\omega\in\spaceLoop{\H}
	:
	\edges{\omega}\setminus\edges{\domain}
	=
	\edges{{\omega'}}\setminus\edges{\domain}
\}.
\]
Let $n,x > 0$.  The \emph{loop~\O{n} model} on~$\domain$ with \emph{edge-weight}~$x$ and boundary condition~$\omega'$ is the probability measure $\muLoop_{\domain,n,x}^{\omega'}$ supported on $\spaceLoop{\H;\domain;{\omega'}}$ defined by
\[
	\muLoop_{\domain,n,x}^{\omega'}(\omega) = \tfrac{1}{Z_{\domain,n,x}^{\omega'}} \cdot  n^{\ell(\omega;\domain)}\cdot x^{|\omega\cap\domain|},
\]
where~$\ell(\omega;\domain)$ is the number of loops in $\omega$ intersecting $\vertices{\domain}$,
$|\omega\cap\Omega|$ is the number of edges in $\omega\cap\Omega$, and~$Z_{\domain,n,x}^{\omega'}$ is a normalizing
	constant (called the {\em partition function}) that
	renders~$\muLoop_{\domain,n,x}^{\omega'}$ a probability measure.

 Among particular cases of the loop~\O{n} model are the Ising model ($n=1$), $1/2$ Bernoulli site percolation ($n=x=1$), dimer model ($n=1,x=\infty$), integer-valued 1-Lipschitz function ($n=2$); see Fig.~\ref{fig:phase-diagram}.
There are heuristic connections to the spin~\O{n} model when~$n$ is integer; see~\cite{PelSpi17} for a survey.

The above definition can be extended to the infinite volume in a standard way using the Dobrushin, Landford and Ruelle (DLR) formalism; see~\cite{GlaLam25} for more details. We endow the set of loop configurations $\spaceLoop{\H}$ with the $\sigma$-algebra generated by cylinder events. A measure $\mu$ on this space is a \emph{Gibbs measure}
	if, for any domain $\domain$
	and for $\mu$-almost every ${\omega'}$,
	the measure $\mu$ conditional on $\{\omega\in\spaceLoop{\H;\domain;{\omega'}}\}$
	equals $\muLoop_{\domain,n,x}^{\omega'}$.

A Gibbs measure $\mu$ is called translation-invariant if 
\[
\mu(A) = \mu (\gamma^{-1} A)
\] 
for any translation $\gamma$ of the hexagonal lattice and any event $A$. 
For~$r\in\N$, define~$\ballloop_r$ as the subgraph of~$\H$ induced by the vertices belonging to the faces of~$\H$ centered at~$k+\ell
e^{i\pi/3}$ with~$|k\pm \ell| \leq r$.
Define also an annulus~$A_r := \ballloop_{2r}\setminus \ballloop_r$.

\begin{theorem}\label{thm:LoopOn}
	Fix $(n,x) \in [1,2] \times [1/\sqrt{2},1]$, and let $\mathbb{P}$ be a translation-invariant Gibbs measure for the loop \O{n} model with edge-weight~$x$. Then, every loop configuration has infinitely many loops surrounding every face, $\mathbb{P}$ almost surely.

	If, in addition, $nx^2 \leq 1$, then for some $c >0$, any loop configuration $\omega'$, and any $r >2$,
	\begin{equation}\label{eq:RSW}
		c\leq \mathbb{P}_{B_{2k},n,x}^{\omega'}[\exists\text{ a loop in $A_k$ surrounding } 0]\leq 1-c.
	\end{equation}
\end{theorem}

\begin{figure}
	\begin{center}
		\includegraphics[width=0.58\textwidth]{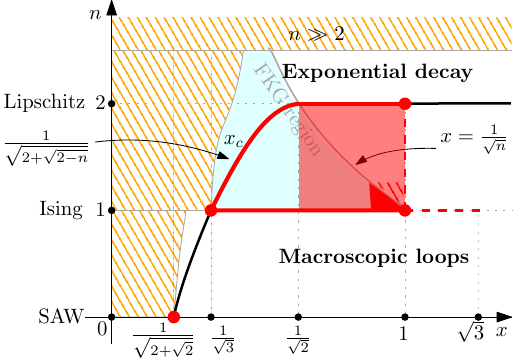}
	\end{center}
	\caption{The phase diagram of the loop \O{n} model: above~$n=2$ and to the left of $x_c(n)$, loop lengths should have exponential tails; below and on the curve~$x_c(n)$, the loops should be macroscopic and converge to $\mathrm{CLE}_\kappa$. Established regions are in orange and red respectively. The current work treats the opaque red rectangle.}
	\label{fig:phase-diagram}
\end{figure}

The Russo--Seymour--Welsh (RSW) type estimates provided by~\eqref{eq:RSW} imply the macroscopic behavior of the loop~\O{n} model: existence of loops at every scale.
Theorem~\ref{thm:LoopOn} is the first result to prove such behavior in a two-dimensional non-perturbative region of parameters. 
Previously, macroscopic behavior has been proven in several cases (see Fig.~\ref{fig:phase-diagram}):
\begin{itemize}
	\item At~$x=x_c(n)$ when~$n\in [1,2]$~\cite{DumGlaPel21},
	\item An area~$n\in [1, 1+\varepsilon], x\in [1-\varepsilon,1/\sqrt{n}]$ around~$n=x=1$~\cite{CraGlaHarPel25},
	\item $n=2, x\in [1/\sqrt{2}, 1]$~\cite{GlaMan21,GlaLam25}.
\end{itemize}
In addition, this is standard for the loop~\O{1} model with~$x\in [1/\sqrt{3},1]$ which corresponds to the Ising model on the triangular lattice at~$T>T_c$ (eg. follows from~\cite{Tas16}). 
We also mention that it has been shown that any translation-invariant Gibbs measure for the loop \O{1} model at~$x=\infty$ (fully-packed case, whose complement is uniform dimers on the hexagonal lattice) either exhibits infinitely many loops around every face or a unique bi-infinite path~\cite{GlaRey24}.
The macroscopic behavior is expected to hold for all $n\in[0,2]$ when~$x\geq x_c(n)$, where
\[
	x_c(n):=\tfrac{1}{\sqrt{2+\sqrt{2-n}}}.
\]
This value was derived first at~$n=2$~\cite{DomMukNie81,Fan72} based on a relation with the Ashkin--Teller model and then for all~$n\in [0,2]$~\cite{Nie82} using the Coulomb gas formalism.

The macroscopic behavior is a weaker version of scale invariance and is in agreement with the conjectured convergence to the conformal loop ensemble (CLE) with an appropriate~$\kappa = \kappa (n,x)$; see~\cite[Section~5.6]{KagNie04}.
The convergence has been proved only for Bernoulli site percolation ($n=x=1$)~\cite{Smi01,Smi01a,CamNew06,KhrSmi21}
and for the critical Ising model~($n=1,\, x=1/{\sqrt{3}}$)~\cite{Smi10,CheSmi11,CheDumHonKemSmi14}.

 When~$n=2$, the loops are level lines of an integer-valued 1-Lipschitz function on the faces of~$\H$.
Theorem~\ref{thm:LoopOn} provides a new proof of its delocalization for all~$x\in [x_c(2),1]$.
Localization is expected for all~$x< x_c(2)$; this has been proven for~$x< x_c(1) + \varepsilon$ (note that~$x_c(1)=1/\sqrt{3}$ and~$x_c(2)=1/\sqrt{2}$).

Besides the use of Theorem~\ref{thm:perco}, the proof of Theorem~\ref{thm:LoopOn} combines elements of the geometric representation in~\cite{GlaLam25}, the `flipping' transformation in~\cite{GlaMan21}, and the defect percolation construction of~\cite{CraGlaHarPel25}. This approach rules out bi-infinite paths even outside the positive association regime. Inside the positive association regime, the dichotomy theorem of~\cite{DumGlaPel21} allows us to upgrade the infinite loop state to a RSW-type result. We mention that~\cite{KohTas23} proves RSW estimates for general positively correlated percolation processes on lattice with $\Z^2$ symmetries. Our approach does not rely on this work, and we do not believe it can be used to imply the conclusions of Theorem~\ref{thm:LoopOn}.

\subsection{Sketch of the proof of Theorem~\ref{thm:perco}}
\label{sec:sketch}

In this subsection, we present the main ideas of the proof of Theorem~\ref{thm:perco}.
For clarity of presentation, here we assume that~$G$ has an embedding with no accumulation points.
It is easy to show that there exists an increasing sequence of simply connected open sets~$\domain_n\subset\C$ that exhausts the plane and whose boundaries are Jordan curves that intersect~$G$ only at vertices; see Lemma~\ref{lem:jordan-domains} for a slightly weaker version in the general setting.

Assume~$N_\infty\geq 1$ a.s.
By tail triviality of~$\sigma$, for any~$\varepsilon>0$ and~$n$ large enough,
\begin{equation}\label{eq:one-arm}
	\P(\partial\domain_n\xlra{\domain_n^c\cap \{\sigma=1\}} \infty) > 1-\varepsilon,
\end{equation}
where the event above states the existence of an infinite open path contained in~$\domain_n^c$ and starting at a vertex on~$\partial\domain_n$.

The main step of the proof is splitting~$\partial\domain_n$ into several arcs, each connected to infinity in~$\domain_n^c$ with probability almost one. Similar approaches appeared in the context of Voronoi percolation on $\mathbb{R}^2$~\cite{Tas16}, and in level-set percolations of one-Lipschitz functions on cubic, vertex-transitive graphs~\cite{Kar23}.

Denote the vertices on~$\partial\domain_n$ by~$v_1, \dots, v_L$, taken in clockwise order.
For each~$i,j\in \{1,\dots,L\}$, let~$p_{i,j}$ denote the probability that~$\{v_i,v_{i+1},\dots,v_j\}$ does {\em not} connect to infinity in~$\domain_n^c$.
Clearly, $p_{1,i}$ is decreasing and $p_{1,L} < \varepsilon$.
Define
\[
	i^*:= \min \{i\in \{1,\dots,L\} \colon p_{1,i} < \sqrt{2\varepsilon}\}.
\]
By the FKG inequality and the definition of~$n$, 
\[
	p_{1,i^*-1}\cdot p_{i^*,L} \leq p_{1,L} < \varepsilon.
\]
By the minimality of~$i^*$, we have that~$p_{1,i^*-1} \geq \sqrt{2 \varepsilon}$, whence~$p_{i^*,L}< \sqrt{\varepsilon/2}$.
Since~$\sigma$ is dominated by~$1-\sigma$, we have that
\[
	p_{i^*,i^*} \geq \P(\sigma(v_{i^*}) = 0)  \geq 1/2.
\]
Using the FKG inequality again, we get
\[
	p_{i^*+1,L} \leq \tfrac{p_{i^*,L}}{p_{i^*,i^*}}< \sqrt{2\varepsilon}. 
\]
Thus, the arcs~$\{1,\dots,v_{i^*}\}$ and~$\{v_{i^*+1},\dots,v_L\}$ both connect to infinity in~$\domain_n^c$ with probability at least~$1-\sqrt{2\varepsilon}$, as desired.

Choosing~$\varepsilon>0$ small enough and iterating this process, for any integer~$k$, we can split~$\partial\domain_n$ (for~$n$ large enough) into~$2k$ arcs, each connecting to infinity by open paths with probability almost one.
Since~$\sigma$ is dominated by~$1-\sigma$, this statement also holds if we replace open paths by closed ones.
Define~$\mathrm{Arm}_{2k}$ as the event that each arc connects to infinity by both open and closed paths {\em simultaneously}; see Fig.~\ref{fig:catalan}.
By the union bound, the probability of~$\mathrm{Arm}_{2k}$ is almost one; see Lemma~\ref{lem:crossings} for a precise statement.

Finally, once~$\mathrm{Arm}_{2k}$ occurs, planar topology implies that there are at least~$k+1$ open and closed infinite connected components in total; see Lemma~\ref{lem:catalan} and Figure~\ref{fig:catalan}.
If~$\sigma$ and~$1-\sigma$ have the same law (eg. Bernoulli site percolation at~$p = 1/2$), then
\[
	1 - \varepsilon < \P(N_\infty(\sigma) + N_\infty(1-\sigma) \geq k+1) \leq 2 \cdot \P (N_\infty(\sigma) > (k+1)/2).  
\]
Taking the intersection over~$k$, we get that~$N_\infty(\sigma) = \infty$ with probability at least~$1/2-\varepsilon$ and, hence, with probability one, by tail triviality.

If we only have that~$\sigma$ is dominated by~$1-\sigma$, we use Strassen's theorem that states existence of a monotone coupling~$(\sigma,\tau)$, where~$\tau$ has the same law as~$1-\sigma$ and~$\sigma \leq \tau$, almost surely.
Applying the above reasoning for components in~$1-\tau$ instead of~$1-\sigma$ completes the proof; see Lemma~\ref{lem:catalan}.

\subsection*{Acknowledgements}
The project was inspired by a discussion between the second author and Ron Peled about no-coexistence in non translation-invariant percolation processes on $\mathbb{Z}^2$. We thank him for sharing his insights. We also thank Diederik van Engelenburg, Russell Lyons, and Gabor Pete for  fruitful discussions on unimodular graphs.
This collaboration started during an open problem session at a conference at the ETH Zurich in 2024, and we thank Hugo Duminil-Copin and Vincent Tassion for organizing this wonderful event.
Parts of the project were completed during the visits of AG and MH to Northeastern University and the  University of Innsbruck, and we would like to thank these institutions for their hospitality.

This research was funded in part by the Austrian Science Fund (FWF) 10.55776/P34713.

\section{Planar percolation models: proof of Theorem~\ref{thm:perco}}
\label{sec:perco}

Let $G = (V,E)$ be a planar, connected, locally finite graph. We will also assume $G$ is simple, as multiple edges and self-loops do not affect site percolation. Given an embedding $\phi: V \to \mathbb{S}^2$, we call~$a\in \mathbb{S}^2$ an {\em accumulation point} if its arbitrarily small neighborhoods intersect infinitely many edges of~$G$.
Denote the set of accumulation points by~$A$.
We call~$\phi$ {\em well-separated} if~$\phi(G)$ is disjoint from $A$.
In Proposition~\ref{prop:embedding} below, we show that~$G$ always has a well-separated embedding.

We start by introducing useful notation.
For $S \subseteq \mathbb{S}^2$, let $G_{\phi}[S]$ be the subgraph of $G$ induced by the preimage of $\phi(V(G)) \cap S$. 
For $R\subset E$ finite and connected, we define faces of $R$ under $\phi$ to be the path connected components of $\mathbb{S}^2 \setminus \phi(R)$.
Since $\phi(R)$ is the union of finitely many simple paths, the faces of~$R$ are simply connected open sets.

\begin{lemma}\label{lem:sausage}
	Let $G$ be a planar, connected, locally finite graph and $\phi$ an embedding of $G$ into $\mathbb{S}^2$. Consider any $R \subset E$, a finite, connected set of edges, with faces $(F_1, \dots, F_n)$. Then, there exists $\eps >0$ sufficiently small and an embedding $\phi' = \phi'(\phi,R)$ such that~$\phi = \phi'$ on~$R$ and $d (\phi'(G_\phi[F_j]), \phi(R)) > \eps$ 
	for every $1 \leq j \leq n$, where $d(\cdot, \cdot)$ is the Euclidean distance.
\end{lemma}

\begin{proof}
We begin by setting $\phi' = \phi$ on $R$, and construct the rest of the embedding in a piecewise fashion. For $j \in \{1,\dots, n\}$, denote $G_j := G_{\phi}[F_j]$ and let $E_j= \{e_{j,1}, \dots, e_{j,m_j}\}$ be the set of edges between $V(R)$ and $V(G_j)$. 
Define $r_j$ as one half of the minimal length of $\phi(e_{j,k})$.
For every $k \in \{1, \dots, m_j\}$, write~$\phi(e_{j,k})$ as a concatenation of two closed paths~$\gamma_{j,k}^{1}\circ\gamma_{j,k}^{2}$, where $\gamma_{j,k}^1$ starts in~$V(R)$ and has length $r_j$; define $y_{j,k}:=\gamma_{j,k}^{1}\cap\gamma_{j,k}^{2}$.

Let $F_j' = F_j \setminus (\gamma_{j,1}^1\cup\dots\cup\gamma_{j,m_j}^1)$. 
This is an open, simply connected set. Its boundary is given by the union of a (possibly empty) Jordan curve and a forest. Such cycle-rooted forests embedded in the plane can be parameterized by an oriented path that covers each edge either once or twice.
We relabel the edges of $E_j$ so that $\{y_{j,1}, y_{j,2}, \dots  y_{j,m_j}\}$ appear in clockwise order on this path.

Pick some $x_j$ in the interior of $F_j'$, and set $\eps_j < d(x_j, \partial F_j')/2$. By the Riemann mapping theorem, there exists a biholomorphic map $\psi_j$ taking $F_j'$ to $B_{\eps_j}(x_j)$, the Euclidean ball of radius $\eps_j$ around $x_j$ in $\mathbb{S}^2$. 
The embedding of $G_j$ under $\phi\circ\psi_j$ is at distance at least~$\eps_j$ to $\phi(R)$.
By holomorphicity of~$\psi_j$, the points~$z_{j,k}:=\psi_j(y_{j,k})\in\partial B_{\eps_j}(x_j)$ are clockwise ordered.

Now, consider the topological annulus $T_0 = F_j' \setminus B_{\eps_j}(x_j)$. Since $T_0$ is simply connected we can find a simple path $\ell_1$ from $y_{j,1}$ to $z_{j,1}$ in $T_0$; similarly, find a path~$\ell_2$ from $y_{j,2}$ to $z_{j,2}$ in $T_0 \setminus \ell_1$. We now define the domain $T_k$ and the path $\ell_k$ iteratively. Let $J_k$ be the Jordan curve given by concatenating the clockwise-oriented path from $y_{j,k}$ to $y_{j,1}$ along $\partial F_j'$, $\ell_1$, the counterclockwise-oriented path from $z_{j,1}$ to $z_{j,k}$ along $\partial B_{\eps_j}(x_j)$, and $\ell_k$. Let $T_k$ be the component of $\mathbb{S}^2\setminus J_k$ which is entirely contained in $F_j'$. This component is disjoint from $\{\ell_2, \dots, \ell_{k-1}\}$, and contains $y_{j,k+1}$ and $z_{j,k+1}$ in its boundary. We set $\ell_{k+1}$ to be an arbitrary simple path from $y_{j,k+1}$ to $z_{j,k+1}$ in $T_k$. For every $e_{j,k} \in E_j$, we set $\phi'(e_j)$ to be the concatenation of $\gamma_{j,k}^1$, $\ell_k$, and $\psi_j(\gamma_{j,k}^2)$. This produces an embedding of $G_j$, $R$, and all edges between the two graphs. Repeating this procedure for all faces and setting $\varepsilon = \min_j \varepsilon_j$ completes the proof. 
\end{proof}

\begin{figure}[t]
\centering 
\begin{subfigure}[]{0.3 \textwidth}
\includegraphics[scale=0.35,page=1]{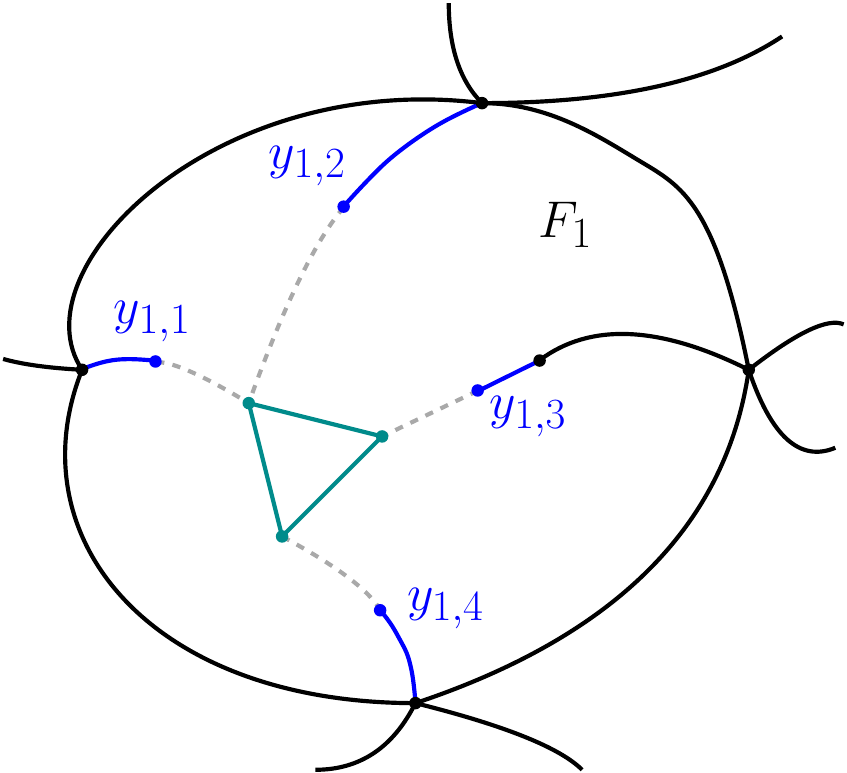} \end{subfigure} {\color{white} cec} 
\begin{subfigure}[]{0.3 \textwidth}
\includegraphics[scale=0.35,page = 2]{sausage3.pdf} 
\end{subfigure} {\color{white} cec}
\begin{subfigure}[]{0.3 \textwidth}
\includegraphics[scale=0.35,page = 3]{sausage3.pdf}
\end{subfigure}
\caption{\textsc{Left}: A portion $\phi(R)$ bounding the face $F_1$  in black, the graph $G_1$ in cyan, and the edges of $E_1$ in dashed gray. The half edges $\gamma_{1,k}^1$, with $y_{1,k}$ as their endpoints, in blue. 
\textsc{Middle}: The new embedding of $G_1$ in a ball contained in the interior in green. The image of the half edges $\gamma_{1,k}^2$ with their endpoints $z_{1,k}$ in burgundy. 
\textsc{Right}: A choice for the paths connecting the first three sets of endpoints, and the domain $T_3$ in purple. }
\end{figure}

\begin{proposition}\label{prop:embedding}
	Any planar, connected, locally finite graph has a well-separated embedding.
\end{proposition}

\begin{proof}
Let $\phi_0$ be an arbitrary planar embedding, and let $\{e_i\}$ be an enumeration of the edges of $G$ such that $R_k = \{e_1, \dots, e_k\}$ is a connected set of edges for every $k$. We now apply Lemma~\ref{lem:sausage} iteratively: we begin with $(\phi_0, R_1)$, and set $\phi_k = \phi'(\phi_{k-1},R_k)$ for every~$k$. For any $j \geq k$, $\phi_j(e_k) = \phi_k(e_k)$, so the image of every edge stabilizes after finitely many applications of this process. Therefore, the limiting map $\phi = \lim_k \phi_k$ is well-defined on all of $G$. Furthermore, the set $
\{x : d(x, R_k) \leq \varepsilon_k\}$ cannot contain any accumulation points of~$\phi$ for any~$k$, and thus $A$ is disjoint from $\phi(G)$, as required. 
\end{proof}

For the rest of the section, we consider a planar, connected, locally finite graph~$G$ with a fixed well-separated embedding $\phi$ and identify vertices and edges with their images.
Let~$\sigma$ be a tail trivial site percolation process on~$G$.
For any sets $S,T, \domain \subseteq \mathbb{S}^2$ and~$U\subset V$, define
\[
\{S \xleftrightarrow{\domain \; \cap\;  U} T\} := \{\exists \text{ connected component } \mathcal{C} \subseteq (\domain \cap U ) \text{ with } \overline{\mathcal{C}} \cap S, \overline{\mathcal{C}} \cap T  \neq \emptyset\},
\] 
where~$\overline{\calC}$ is the topological closure of~$\calC$. 
If $S$ and $T$ are finite sets of vertices, the above agrees with the usual, combinatorial notion of connectivity.
When~$\domain = \mathbb{S}^2$ and~$U=\{\sigma=1\}$, we will omit the superscript for brevity.  

We call $a \in A$ an {\em active} point of a site percolation $\sigma$ if $v  \xleftrightarrow{} a$ for some~$v\in V$. 
Heuristically, this means that some infinite component of $\sigma$ must approach $a$, in the weak sense outlined above. Since active points must be accumulation points, the event that~$a\in A$ is an active point of~$\sigma$ is in the tail sigma-algebra, and thus has probability zero or one. 
We note that $a$ being active does not necessarily imply the existence of an infinite path that approaches $a$ in $\sigma$; it is possible that there exists a connected component of $\sigma$ which contains a sequence of infinite paths $\{\gamma_k\}$, each approaching a point $a_k \neq a$, such that $a = \lim a_k$. 
We now show that~$\sigma$ has an active accumulation point if~$1\leq N_\infty(\sigma) <\infty$. 
Note that this can be false if $N_\infty(\sigma) =\infty$; indeed, the usual embedding of the $3$-regular tree in the unit disk has no active points at any $p \in (1/2,1)$.

\begin{lemma}\label{lem:activating-accum-pt}
	Let $G$ be a planar, connected, locally finite graph.
	Consider a tail trivial site percolation process~$\sigma$ on~$G$.
	Suppose that $1 \leq N_{\infty}(\sigma) < \infty$ a.s.
	Then, there exists an active point $a \in \mathbb{S}^2$, almost surely.  
\end{lemma}

\begin{proof}
	For each~$k\in \N$, draw the square lattice~$\L_k$ of mesh-size~$2^{-k}$ on~$\bbS^2$, where we take each face of the lattice to be a closed set. Assume $N_\infty(\sigma) \geq 1$. For each face $F$ of~$\L_k$, the union of~$\{v \xleftrightarrow{}  A \cap F\}$ over~$v\in V$ is a tail event; since~$\L_k$ has finitely many faces, there must exist at least one face such that $\{v \xleftrightarrow{}  A \cap F\}$ occurs for some~$v\in V$, almost surely. Repeating this for each k, we can produce a decreasing sequence of faces $F_k$ of lattices~$\L_k$, and a connected component of $\sigma$ which approaches some point of $A \cap F_k$. 
	Since $N_{\infty}(\sigma) < \infty$, we can find a single connected component $\mathcal{C}$ and some subsequence of faces $F_{k_\ell}$ such that $\mathcal{C}$ connects to $A \cap F_{k_\ell}$ for every $\ell$.
	Then, $\bigcap_\ell F_{k_\ell}$ is an active point.
\end{proof}

We now construct a sequence of closed sets $\{\domain_n\}$ exhausting~$G$ that, in the general case, will replace Euclidean balls used in Section~\ref{sec:sketch}; see Lemma~\ref{lem:jordan-domains} below for the exact list of properties. We also construct an associated sequence of vertex sets $\{\calS_n\}$ which will essentially function as cut sets (see Item 3 in Lemma~\ref{lem:jordan-domains}).
We start by fixing an active accumulation point~$a\in\bbS^2$ and a vertex $\rho$ of~$G$. 
Let $\Lambda_n(\rho)$ be the combinatorial ball about $\rho$ of radius $n$ --- that is, the graph induced by $\{v \in V(G): d_G(\rho,v) \leq n\}$, where $d_G(\cdot,\cdot)$ is the metric on $G$. Similarly, we define the boundary $\partial \Lambda_n(\rho) = \{v \in V(G): d_G(\rho,v) = n\}$. Now let $Q_n$ be the connected component of $a$ in $\bbS^2 \setminus \Lambda_n(\rho)$, and $\domain_n := \bbS^2 \setminus Q_n$.
Define~$J_n\subseteq \Lambda_n(\rho)$ as the topological boundary of $\domain_n$, and~$\calS_n := J_n \cap \partial \Lambda_n(\rho)$.

\begin{lemma}\label{lem:jordan-domains}
	The sequence of pairs of sets $\{\domain_n, \calS_n\}$ satisfies the following properties:
	\begin{enumerate}
	\item $\domain_n \subseteq \domain_{n+1}$, $\calS_n\cap \calS_{n+1} = \emptyset$, and $G \subseteq \bigcup_{n} \domain_n$,
 	\item $\domain_n$ is a closed, simply connected subset of~$\bbS^2$,  its boundary $J_n$ contains $\calS_n$ and can be continuously parametrized in such a way that no vertex of~$\calS_n$ is covered twice,
	\item for any vertex~$v \in \Omega_n$, any infinite connected component of $\{\sigma=1\}$ whose closure contains both $v$ and $a$ must contain a vertex in $\calS_n$.
\end{enumerate}
\end{lemma}

\begin{proof}
	For Item~$1$, it is clear that $\Lambda_n(\rho) \subseteq \Lambda_{n+1}(\rho)$.
	This implies $Q_n \supseteq Q_{n+1}$ and~$\domain_{n}\subseteq \domain_{n+1}$ as requested. 
	Also, $\calS_n\subseteq \partial \Lambda_n(\rho)$ and~$\partial \Lambda_n(\rho) \cap \partial \Lambda_{n+1}(\rho) = \emptyset$, whence~$\calS_n\cap \calS_{n+1} = \emptyset$.
	Since $G$ is connected, the combinatorial balls~$\Lambda_n(\rho)\subseteq \domain_n$ exhaust $G$, whence~$G \subseteq \bigcup_n \domain_n$.
	
	For Item~$2$, note that $Q_n$ is open and path connected, whence $\domain_n$ is closed and simply connected. 
	Their common boundary~$J_n$ is the union of all (finitely many) edges that border~$Q_n$.
	By picking a starting point on $J_n$ and traversing this boundary in a clockwise manner, one obtains acontinuous parametrization $\gamma_n: \bbS^1 \to J_n$.  Given distinct~$x,y\in \bbS^1$, define~$(x,y)$ as the clockwise arc excluding the endpoints. If $\gamma_n(x)=\gamma_n(y)=v\in V$, then both~$\gamma_n((x,y))\cap V$ and~$\gamma_n((y,x))\cap V$ are nonempty (since~$G$ is simple).
	Now, define~$\Lambda_n(\rho)^*$ as the dual graph of~$\Lambda_n(\rho)$ with a vertex~$q_n^*$ corresponding to~$Q_n$.
	If the first and last edges in~$\gamma_n((x,y))$ coincide, then define~$\calC$ as the self-loop in~$\Lambda_n(\rho)^*$  at~$q_n^*$ crossing this edge.  
	Otherwise, denote these edges by~$e_1,e_2$, their duals by~$e_1^*,e_2^*$ and construct~$\calC$ as follows:  start at~$q_n^*$, trace a path along $e_1^*$ and $e_1$ to $v$, then along $e_2$ and $e_2^*$ back to~$q_n^*$.
	In either case, $\calC$ is a Jordan curve that separates~$\gamma_n((x,y))$ from~$\gamma_n((y,x))$ in~$\Lambda_n(\rho) \setminus \{v\}$.
	Thus, $v$ is a cut-vertex of~$\Lambda_n(\rho)$.
	In particular, $v\not\in \calS_n\subset \partial\Lambda_n(\rho)$ since no vertex in~$\partial\Lambda_n(\rho)$ can be a cut-vertex.
	
	For Item~$3$, we can find~$\varepsilon>0$ such that ball of radius~$\varepsilon$ around~$a$ does not intersect~$\domain_n$.
	By our definition of connectivity, there exists a vertex~$u$ in this ball that is connected to~$v$ in~$\{\sigma=1\}$.
	Follow this path from~$v$ to~$u$ and let~$w$ be the last vertex on the path that is contained in~$\domain_n$.
	Then, $w\in J_n$, and hence~$w\in \calS_n$.
\end{proof}

As a corollary of Lemma~\ref{lem:jordan-domains}, we see that
\[
\bigcup_{n} \{\calS_n  \xleftrightarrow{\domain_n^c \; \cap\; \{\sigma=1\}} a\} = \{a \text{ is an active point of } \sigma \}.
\]
Since the sequence of events is increasing, for any $\eps >0$ and $n$ large enough, we have

\begin{equation}\label{eq:VeryHighProbability}
\mathbb{P}\left(\calS_n   \xleftrightarrow{\domain_n^c \; \cap \; \{\sigma=1\}} a \right) > 1- \eps.
\end{equation}
 Using a parametrization from Lemma~\ref{lem:jordan-domains}, we denote vertices of~$\calS_n $ by~$v_1^n$, \dots, $v_L^n$ in clockwise order.
Given $1 \leq i_1 < \dots < i_{2k} \leq L$, we can partition $\calS_n$ into $2k$ arcs, setting $\mathrm{Arc}_j := \{v_{i_{j-1} +1}, \dots, v_{i_{j}}\}$ (where $i_0 := i_{2k}$).
Define the following arm event:
\[
\mathrm{Arm}_{n}^{2k}(i_1, \dots, i_{2k}) := \bigcap_{j=1}^{2k} \{\mathrm{Arc}_j\xleftrightarrow{\domain_n^c \; \cap \; \{\sigma=1\}} a\} \cap \{\mathrm{Arc}_j \xleftrightarrow{\domain_n^c \; \cap \; \{\sigma=0\}} a\}
\]   
The event is symmetric with respect to replacing~$\sigma$ by~$1-\sigma$ and implies the existence of~$2k$ alternating open and closed crossings; see Fig.~\ref{fig:catalan}.

\begin{lemma}\label{lem:crossings}
	Let~$\varepsilon>0$ and~$k\in \N$. Let $\sigma$ be a tail trivial site percolation that is stochastically dominated by $1 - \sigma$. Then, for~$n\in\N$ large enough, there exist~$1\leq i_1<\dots<i_{2k}\leq L$ such that,
	\[
		\P(\sigma\in \mathrm{Arm}_n^{2k}(i_1,\dots,i_{2k})) > 1-\varepsilon.
	\]

\end{lemma}

\begin{proof}
	Fix $\eps' = (\eps/8k)^{2k}$. We follow the argument detailed in Section~\ref{sec:sketch}. For each~$i,j\in \{1,\dots,L\}$, let~$p_{i,j}$ denote the probability that~$\{v_i,\dots,v_j\}$ (in clockwise order) does {\em not} connect to~$a$ in~$\domain_n^c$.
	Again, $p_{1,i}$ is decreasing and~$p_{1,L}< \eps'$, thanks to~\eqref{eq:VeryHighProbability}.
	Define
	\[
		i_1:=\min\{i\in \{1,\dots,L\}\colon p_{1,i} \leq \eps/4k\}.
	\]
	As in Section~\ref{sec:sketch}, by the FKG inequality and the domination of~$\sigma$ by~$1-\sigma$, we get
	\[
		p_{i_1+1,L} \leq p_{1,L}/p_{1,i_1} \leq (\eps/8k)^{2k-1}.
	\]
	Iterating this~$2k-1$ times, we get~$1<i_1<\dots<i_{2k}=L$ such that $p_{i_{j-1},i_j} \leq \eps/4k$ for each~$j=1,\dots, 2k$.
	By the stochastic domination, the same holds also for crossings of closed vertices.
	The lower bound on the probability of $\mathrm{Arm}_n^{2k}(i_1,\dots,i_{2k})$ then follows applying the union bound to the complement of the event.
\end{proof}

\begin{lemma}\label{lem:catalan}
	Let~$n,k\in \N$, $1\leq i_1<\dots<i_{2k}\leq L$  and~$\sigma,\tau\in \mathrm{Arm}_n^{2k}(i_1,\dots,i_{2k})$ such that~$\sigma\leq \tau$.
	Then, $N_\infty(\sigma) + N_\infty(1-\tau) \geq k+1$.
\end{lemma}

\begin{figure}[t]
\begin{center}
\begin{subfigure}[]{0.3 \textwidth}
\includegraphics[scale=0.52,page=1]{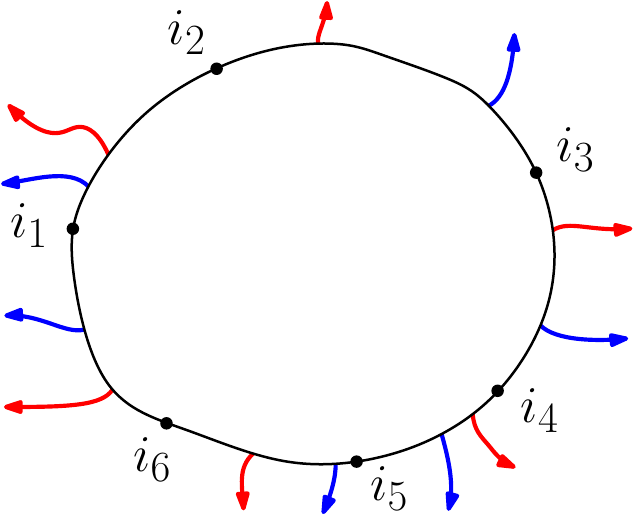} \end{subfigure} {\color{white} ce ce ce ce} 
\begin{subfigure}[]{0.3 \textwidth}
\includegraphics[scale=0.52,page = 2]{catalan2.pdf}  
\end{subfigure}
\end{center}
 \caption{\textsc{Left}: a configuration in $\mathrm{Arm}_n^{2k}(i_1, \dots, i_{2k})$ for~$k=3$ with open (blue) and closed (red) infinite paths. \textsc{Right}: a coloring of the interior of the domain. As shown by the green circles, planar topology restricts the possible connectivity of the infinite paths for any coloring. In the illustrated example, there are five open and closed infinite components in total, which is greater than~$k+1=4$.}
\label{fig:catalan}
\end{figure}

\begin{proof} 
	For~$m>n$, define~$C_{n,m}(\sigma)$ (resp.~$C_{n,m}(1-\tau)$) as the number of connected components of $\domain_m \cap \{\sigma=1\}$ (resp.~$\domain_m \cap \{\tau=0\}$) which contain a crossing from $\calS_n$ to $\calS_m$ in the annular domain $\domain_m \setminus \domain_n$.
	It is enough to show that~$C_{n,m}(\sigma) + C_{n,m}(1-\tau) \geq k+1$ for each~$m>n$.
	We prove this by induction on $k$.
	
	By the arm event, for~$j=1,\dots, k$, we can find~$b_j\in \mathrm{Arc}_{2j-1},w_j\in \mathrm{Arc}_{2j}$ such that~$\sigma$ (resp.~$\tau$) contains an open (resp. closed) crossing in~$\domain_m \setminus \domain_n$ that starts at~$b_j$ (resp.~$w_j$).
	If all the points~$b_1,\dots,b_k$ belong to distinct open connected components of~$\sigma$ in~$\domain_m$ (which is the case always when~$k=1$), then the statement is trivial since we always have at least one connected component of closed vertices in~$\tau$ that contains a crossing.
	
	Without loss of generality, assume that~$k\geq 2$ and that~$b_1$ is connected to~$b_i$ in~$\domain_m\cap \{\sigma=1\}$.
	Let~$\gamma_1$ and~$\gamma_i$ be open paths in~$\sigma$ linking~$b_1$ and~$b_i$ to~$\calS_m$ in~$\domain_m\setminus \domain_n$.
	Since~$\sigma\leq \tau$, both paths~$\gamma_1$ and~$\gamma_i$ are also open in~$\tau$.
	Due to the existence of closed crossings in~$\tau$ starting at~$w_1$ and at~$w_i$, the paths~$\gamma_1$ and~$\gamma_i$ are disjoint by the planar duality.
	Let~$\gamma_{1i}$ be an open path in~$\sigma$ in~$\domain_m$ connecting~$\gamma_1$ to~$\gamma_i$ and disjoint from them, besides its endpoints.
	Again, since~$\sigma\leq \tau$, the path~$\gamma_{1i}$ is also open in~$\tau$.
	Then, in~$\domain_m$,
	\begin{itemize}
		\item there is no closed crossing in~$\tau$ from~$(w_1,\dots,w_{i-1})$ to~$(w_{i+1},\dots,w_k)$;
		\item any open connected component of~$\sigma$ intersecting both~$(b_2,\dots,b_{i-1})$ and~$(b_{i+1},\dots,b_k)$ contains~$b_1$.
	\end{itemize}
	Now, we apply the induction hypothesis to the restrictions of~$\sigma$ and~$\tau$ to the parts in which~$\gamma_1\cup \gamma_{1i}\cup\gamma_i\cup \partial \domain_m$ splits~$\domain_m$.
	Adding the numbers of clusters in different parts and subtracting one common cluster containing~$b_1$, we get~$C_{n,m}(\sigma) + C_{n,m}(1-\tau) \geq k+1$.
\end{proof}

We are now ready to finish the proof of Theorem~\ref{thm:perco}.

\begin{proof}[Proof of Theorem~\ref{thm:perco}]
	Consider a well-separated embedding of~$G$ in~$\bbS^2$ provided by Proposition~\ref{prop:embedding}.
	By tail triviality, it is enough to exclude that~$1 \leq N_{\infty}(\{\sigma=1\}) < \infty$, a.s.
	Since~$\sigma$ is stochastically dominated by~$1-\sigma$, Strassen's theorem gives a coupling between~$\sigma$ and a random variable~$\tau$ with the same law as~$1-\sigma$ such that~$\sigma\leq \tau$ a.s. 
	Picking $\varepsilon < 1/2$, we can apply Lemmata~\ref{lem:jordan-domains} and~\ref{lem:crossings} to find a domain~$\domain_n$ and~$1\leq i_1<\dots<i_{2k}\leq L(n)$ such that~$\sigma \in \mathrm{Arm}_n^{2k}(i_1,\dots,i_{2k})$ with probability at least~$1-\varepsilon$.
	Since~$\mathrm{Arm}_n^{2k}(i_1,\dots,i_{2k})$ is symmetric in $\sigma$ and $1 - \sigma$, we get that $\tau \in \mathrm{Arm}_n^{2k}(i_1,\dots,i_{2k})$ with probability at least $1 - \varepsilon$, as well; see Fig.~\ref{fig:catalan}.
	By the union bound,
	\[
		\P(\sigma, \tau\in \mathrm{Arm}_n^{2k}(i_1,\dots,i_{2k})) \geq 1-2\varepsilon.
	\]
	On this event, by Lemma~\ref{lem:catalan}, we get that~$N_\infty(\sigma) + N_\infty(1-\tau) \geq k+1$.
	Thus,
	\[
		\P(N_\infty(\sigma) + N_\infty(1-\tau) \geq k+1) \geq 1-2\varepsilon.
	\]
 Since~$1-\tau$ has the same law as~$\sigma$, we get that $N_\infty(\sigma)\geq (k+1)/2$ with probability at least~$1/2-\varepsilon$ for any~$k$. This means that~$N_\infty(\sigma)=\infty$ with positive probability, producing the desired contradiction.
\end{proof}

\section{Proof of Corollaries~\ref{cor:p-c} and~\ref{cor:divide-color}}
\label{section:divide-color}

Both Corollaries~\ref{cor:p-c} and~\ref{cor:divide-color} follow from Theorem~\ref{thm:perco} if one can rule out the existence of infinitely many infinite connected components.
A beautiful argument of Burton and Keane~\cite{BurKea89} treats a large class of bond percolations on transitive amenable graphs.
Here we adapt it to the setting of site percolations on unimodular invariantly amenable graphs.
Note that a site percolation might not have any trifurcation point due to combinatorial constraints (e.g. the Kagome lattice).
In order to resolve this issue, we sample an auxiliary uniform spanning forest~$\tau$ on top of our site percolation~$\sigma$ and show that~$\tau$ does have trifurcation points (this is reminiscent to~\cite[Lemma~$7.7$]{LyoPer17}).
We then extend the argument to unimodular graphs using~\eqref{eq:MTP} in a fairly standard way.
We provide all details of the proof of Corollary~\ref{cor:p-c} and then explain how to modify it to the more general setting of Corollary~\ref{cor:divide-color}.

\begin{proof}[Proof of Corollary~\ref{cor:p-c}]
Let $\sigma$ be Bernoulli site percolation of parameter $p \leq 1/2$. The statement is trivial at $p=0$, so we will assume $p >0$. 
Assume also that $N_\infty(\{\sigma=1\}) = \infty$, almost surely, to get a contradiction.

Denote by~$\mathsf{Ber}_G$ the law of~$\sigma$ sampled from~$(G,\rho)$.
By the above, $\mathsf{Ber}_G(N_\infty(\{\sigma=1\}) = \infty)=1$, for almost every realisation of~$(G,\rho)$.
We fix any such realisation of~$(G,\rho)$.
For~$r \in \mathbb{N}$ and a vertex~$v$ of~$G$, let~$\Lambda_r(v)$ be the set of vertices in $G$ at (combinatorial) distance at most $r$ from $v$.
Define $\mathrm{TriComp}_r$ to be the event that three distinct infinite open clusters in~$\Lambda_{r-1}(\rho)^c$ intersect $\partial \Lambda_r(\rho)$.
By the continuity of measure, $\mathsf{Ber}_G(\mathrm{TriComp}_r)>1/2$ for~$r$ large enough.
Fix such~$r$.
Using independence, we get
\begin{equation}\label{eq:three-clust-sigma}
	\mathsf{Ber}_G(\mathrm{TriComp}_r \cap \{\sigma \equiv 1 \text{ on } \Lambda_r(\rho) \}) \geq \tfrac12\cdot p^{|\Lambda_r(\rho)|}.
\end{equation}
Fix any realization of~$\sigma\in\mathrm{TriComp}_r$ which is open on the entire ball~$\Lambda_r(\rho)$.

For~$n\in \N$, consider the subgraph of $G$ induced by $\{\sigma = 1\}\cap \Lambda_n(\rho)$ and denote its connected component containing~$\rho$ by~$\calC_{n,\rho}$.
Define~$\mathsf{UST}_{n,G,\sigma}$ as the uniform spanning tree measure on~$\calC_{n,\rho}$.
It is standard that the UST satisfies a negative correlation inequality, whence the limit of~$\mathsf{UST}_{n,G,\sigma}$  as~$n$ tends to infinity exists; see~\cite{BenLyoPerSch01}.
The limit is called {\em free uniform spanning forest}, we denote it by~$\mathsf{FUSF}_{G,\sigma}$. Below we do not rely on~\cite[Theorem~$5.13$]{AngHutNacRay18} establishing that the FUSF is connected a.s. since that work assumes that~$\rho$ has a finite expected degree.

Denote a sample from~$\mathsf{FUSF}_{G,\sigma}$ by~$\tau$.
For a vertex~$v\in V$, let $\mathrm{Tri}(v)$ be the event that $v$ is a {\em trifurcation} point for~$\tau$: the cluster~$K_{\tau}(v)$ is infinite, and removing $v$ and all its incident edges from $\tau$ splits $K_{\tau}(v)$ into at least three infinite components.
We claim that
\begin{equation}\label{eq:trifurc-exists}
	\mathsf{FUSF}_{G,\sigma}(\exists v\in \Lambda_r(\rho) \colon \tau\in \mathrm{Tri}(v)) \geq 2^{-|E(\Lambda_r(\rho))|-2} =: \varepsilon >0.
\end{equation}
We start by reducing this to a statement about the UST in finite volume.
For~$s>r$ and $v\in \Lambda_r(\rho)$, let $\mathrm{Tri}_s(v)$ be the event that~$v$ is an {\em $s$-almost trifurcation} point for~$\tau$: $K_{\tau}(v)$ contains a crossing from~$\Lambda_r(\rho)$ to~$\partial\Lambda_s(\rho)$, and removing~$v$ and all its incident edges from~$\tau$, splits~$K_{\tau}(v)$ into several components, at least three of which contain crossings~$\Lambda_r(\rho)$ to~$\partial\Lambda_s(\rho)$.
Clearly, $\mathrm{Tri}_s(v)$ is measurable with respect to edges in~$\Lambda_s(\rho)$.
By the continuity of measure, for~$s$ large enough, if~$v\in \Lambda_r(\rho)$ is an~$s$-almost trifurcation, then it is a true trifurcation with probability at least~$1-\varepsilon$.
Fix any such~$s>r$ and~$N>s$ such that the restrictions of~$\mathsf{UST}_{N,G,\sigma}$ and~$\mathsf{FUSF}_{G,\sigma}$ to~$\calC_{s,\rho}$ are~$\varepsilon$-close in the total variation distance.
Thus, it is enough to show
\begin{equation}\label{eq:almost-trifurc-exists}
	\mathsf{UST}_{N,G,\sigma}(\exists v\in \Lambda_r(\rho) \colon \tau\in \mathrm{Tri}_s(v)) \geq 4\varepsilon.
\end{equation}
We now fix any realization~$\tau_\mathrm{out}$ of~$\mathsf{UST}_{N,G,\sigma}$ restricted to~$E(\Lambda_r(\rho))^c$.
Define~$W$ as the subgraph of~$G$ induced by a subset of vertices of~$\Lambda_r(\rho)$ that contains exactly one vertex in each connected component of~$\tau_\mathrm{out}$ that intersects~$\Lambda_r(\rho)$.
In particular, $W$ contains $\Lambda_{r-1}(\rho)$ and, thus, is connected.
Let~$\tau_\mathrm{in}$ be an arbitrary spanning tree of~$W$.
Then, $\tau_\mathrm{out}\cup \tau_\mathrm{in}$ is a spanning tree of~$\calC_{N,\rho}$.
We now show that~$\tau_\mathrm{out}\cup \tau_\mathrm{in}$ contains an $s$-almost trifurcation point in~$\Lambda_r(\rho)$.
This would readily imply~\eqref{eq:almost-trifurc-exists}, since, by the Markov property, $\mathsf{UST}_{N,G,\sigma}$ conditioned on~$\tau_\mathrm{out}$ and restricted to~$\Lambda_r(\rho)$ is uniform over fitting subsets of edges of~$\Lambda_r(\rho)$, and there are at most~$2^{|E(\Lambda_r(\rho))|} = (4\varepsilon)^{-1}$ such subsets.
Recall that~$\sigma\in \mathrm{TriComp}_r$.
Thus, $W$ must contain three vertices of~$\partial\Lambda_r(\rho)$ belonging to three disconnected crossings in~$\tau_\mathrm{out}$ from~$\Lambda_r(\rho)$ to~$\partial\Lambda_s(\rho)$.
Denote these vertices by~$w_1,w_2,w_3$ and consider the unique paths in~$\tau_\mathrm{in}$ from~$w_3$ to~$w_1$ and~$w_2$.
The last common vertex of these paths is a trifurcation and belongs to~$\Lambda_r(\rho)$.
This finishes the proof of~\eqref{eq:almost-trifurc-exists} and hence~\eqref{eq:trifurc-exists}.

Define 
\[
	f(G,\rho,v) = \mathsf{Ber}_G(\mathsf{FUSF}_{G,\sigma}(\mathrm{Tri}(\rho)))\cdot \1_{v \in \Lambda_r(\rho)}\cdot |\Lambda_r(\rho)|^{-1}.
\]
Applying~\eqref{eq:MTP} and using \eqref{eq:three-clust-sigma} and~\eqref{eq:trifurc-exists} we get that
\begin{align}
	\mathbb{E}[ \mathsf{Ber}_G( \mathsf{FUSF}_{G,\sigma} (\mathrm{Tri}(\rho) ) )] 
	&= \mathbb{E} \bigg[\sum_{v \in \Lambda_r(\rho) } \mathsf{Ber}_G( \mathsf{FUSF}_{G,\sigma} ( \mathrm{Tri}(v) ) )\cdot |\Lambda_r(v)|^{-1}\bigg] \nonumber  \\
	&\geq \mathbb{E} \bigg[\mathsf{Ber}_G \bigg(  \mathsf{FUSF}_{G,\sigma}\bigg( \sum_{v \in \Lambda_r(\rho) } \1_{ \mathrm{Tri}(v)} \bigg) \bigg) \cdot \min_{v \in \Lambda_r(\rho)} |\Lambda_r(v)|^{-1} \bigg] \nonumber \\
	\label{eq:trifurc-averaged}
	&\geq \mathbb{E} \bigg[\tfrac12\cdot p^{|\Lambda_r(\rho)|} \cdot 2^{-|E(\Lambda_r(\rho))|-2} \cdot \min_{v \in \Lambda_r(\rho)} |\Lambda_r(v)|^{-1} \bigg]  =: c > 0.
\end{align}
Take any finitary invariant bond percolation $\omega$ and
let $T(G,\rho,\omega)$ be the number of trifurcation points of $\tau$ in $K_{\omega}(\rho)$. By the standard 3-partition argument of Burton and Keane (see \cite[Lemma 8.5]{Gri99a}), $T(G,\rho,\omega)$ is bounded above by $|\partial K_\omega(\rho)|$.
Now set
\[
	g(G,\rho,v,\omega) = \mathsf{Ber}_G( \mathsf{FUSF}_{G,\sigma} (\mathrm{Tri}(v) ) ) \cdot \1_{v \in K_{\omega}(\rho)}\cdot |K_{\omega}(\rho)|^{-1}.
\]
Applying~\eqref{eq:MTP} for the function~$g$ and using~\eqref{eq:trifurc-averaged}, we get that
\begin{align*}
	\mathbb{E}\left[|\partial K_{\omega}(\rho)|\cdot |K_{\omega}(\rho)|^{-1} \right] 
	 &\geq \mathbb{E}\left[\mathsf{Ber}_G( \mathsf{FUSF}_{G,\sigma}(T(G,\rho,\omega)))\cdot |K_\omega(\rho)|^{-1} \right] \\
	 &= \mathbb{E}\left [\mathsf{Ber}_G( \mathsf{FUSF}_{G,\sigma}(\mathrm{Tri}(\rho) ) ) \cdot \sum_{v\in  V}  \1_{\rho \in K_{\omega}(v)}\cdot |K_{\omega}(v)|^{-1} \right] \geq c.
\end{align*}
Since~$\omega$ is an arbitrary finitary invariant bond percolation, this contradicts the invariant amenability of~$G$.
\end{proof}

\begin{proof}[Proof of Corollary~\ref{cor:divide-color}]
	Conditional on~$(G,\rho)$ and the random partition~$P$ with finite elements, the law of~$\sigma$ satisfies the assumptions of Theorem~\ref{thm:perco}.
	Thus, it is enough to rule out the case that~$N_\infty(\{\sigma=1\}) = \infty$, almost surely.
	Compared to Corollary~\ref{cor:p-c}, the main difference lies in the use of an additional scale~$R$.
	Once~\eqref{eq:trifurc-exists} is proven, we just need to apply~\eqref{eq:MTP} to~$(G,\rho)$ and~$P$ jointly.
	Precisely, we define~$r\in\N$ such that
	\[
		\mu(\mathrm{Ber}_G(\mathrm{TriComp_r}))\geq \tfrac78.
	\]
	For an integer~$R>r$, define~$\mathrm{Inside}_R$ as the event that all elements of~$P$ intersecting~$\Lambda_r(\rho)$ lie inside~$\Lambda_R(\rho)$.
	Since~$P$ has finite elements, almost surely, we can find~$R$ such that
	\[
		\mu(\mathrm{Inside}_R) \geq \tfrac78.
	\]
	Using the two previous displays and the fact that~$\mathrm{Ber}_G(\mathrm{TriComp_r})\leq 1$, we get
	\[
		\mu(\mathrm{Inside}_R, \mathrm{Ber}_G(\mathrm{TriComp_r})\geq \tfrac14) \geq \tfrac12.
	\]

	We now fix a partition~$P$ that satisfies these events.
	There is an additional subtlety that the events~$\mathrm{TriComp}_r$ and~$\{\sigma \equiv 1 \text{ on } \Lambda_r(\rho)\}$ may not be independent.
	Define~$\overline{\mathrm{TriComp}}_{r,R}$ as the intersection of the events~$\{\sigma \equiv 1 \text{ on } \Lambda_r(\rho)\}$ and the existence of three infinite components of~$\{\sigma=1\}\cap \Lambda_{R-1}^c(\rho)$ connected to~$\Lambda_r(\rho)$.
	Taking~$\sigma\in \mathrm{TriComp}_r$ and opening all elements of~$P$ intersecting~$\Lambda_r(\rho)$ gives us a configuration in~$\overline{\mathrm{TriComp}}_{r,R}$.
	Thus, the following analogue of~\eqref{eq:three-clust-sigma} holds:
	\[
		\mathsf{Ber}_G(\overline{\mathrm{TriComp}}_{r,R}) \geq \tfrac14\cdot p^{|\Lambda_r(\rho)|}.
	\]
	Fix any~$\sigma\in \overline{\mathrm{TriComp}}_{r,R}$.
	As in Corollary~\ref{cor:p-c}, define~$\calC_{n,\rho}$, consider the USTs on~$\calC_{n,\rho}$, and define~$\mathsf{FUSF}_{G,\sigma}$ as their limit.
	Defining $s$-almost trifurcation points as before, we pick $s$ such that the following holds with probability at least~$1-\varepsilon$: every partition element of $P$ intersecting $\Lambda_R(\rho)$ is inside $\Lambda_s(\rho)$ and every $s$-almost trifurcation point in~$\Lambda_R(\rho)$ is a true trifurcation point.  We then fix any realization~$\tau_{\mathrm{out}}$ outside of~$E(\calC_{R,\rho})$, and repeat the earlier argument to show the existence of $\tau_\mathrm{in}\subseteq E(\calC_{R,\rho})$ such that~$\tau_\mathrm{in}\cup \tau_\mathrm{out}$ has a $s$-almost trifurcation point inside~$\Lambda_R(\rho)$.  Therefore,
	\[
		\mathsf{FUSF}_{G,\sigma}(\exists v\in \Lambda_R(\rho) \colon \tau\in \mathrm{Tri}(v)) \geq 2^{-|E(\Lambda_R(\rho)|-2}.
	\]
	
	Now it remains to average over realizations of~$(G,\rho)$ and~$P$.
	Since~$P$ is invariant, the joint law of~$(G,\rho)$ and~$P$ satisfies~\eqref{eq:MTP}.
	Define
	\[
		f(G,P,\rho,v) =  \mathsf{Ber}_G(\mathsf{FUSF}_{G,\sigma}(\mathrm{Tri}(\rho)))\cdot \1_{v \in \Lambda_R(\rho)}\cdot |\Lambda_R(\rho)|^{-1}.
	\]
	Applying~\eqref{eq:MTP} as in the proof of Corollary~\ref{cor:p-c}, we obtain
	\[
		\mathbb{E}[\mathrm{Tri}(\rho)] 
		\geq 
		\mathbb{E} \bigg[\tfrac12\cdot \tfrac14\cdot p^{|\Lambda_r(\rho)|} \cdot 2^{-|E(\Lambda_R(\rho)|-2} \cdot \min_{v \in \Lambda_R(\rho)} |\Lambda_R(v)|^{-1} \bigg]  =: c > 0.
	\]
	We complete the proof as in Corollary~\ref{cor:divide-color}: take any finitary invariant bond percolation~$\omega$ and consider the function $g$ (which now depends on the partition~$P$ as well). Since $\omega$ and $P$ are both invariant and independent of one another, we can use Fubini's theorem to apply~\eqref{eq:MTP} and get a uniform lower bound on the expected value of~$|\partial K_{\omega}(\rho)|/|K_{\omega}(\rho)|$, which contradicts the invariant amenability.
\end{proof}

\section{Loop~\O{n} model: proof of Theorem~\ref{thm:LoopOn}}
\label{section:graph-rep}

Each vertex of $\hexlattice$ belongs to precisely one vertical edge.
This gives a natural bipartition of the vertices of~$\hexlattice$ into those at the \emph{top} and those at the \emph{bottom} of a vertical edge.
We define~$\upvert(\hexlattice)$ as the part that consists of top endpoints:
\[
	\upvert(\hexlattice):=
	\{k+\ell e^{i\pi/3} \colon k,\ell \in \Z\}
	-i/\sqrt3
	\subset
	\vertices{\hexlattice}.
\]
We identify~$\upvert(\hexlattice)$ with the triangular lattice obtained from it by connecting nearest neighbors by an edge.
For a domain~$\domain$ and a loop configuration~$\omega$ on~$\hexlattice$, we write
\[
	\upvert(\domain):=(\upvert(\hexlattice)\cap \vertices\domain)\setminus \vertices{\partial\domain};
	\qquad
	\upvert(\omega):=\{v\in V(\hexlattice) \colon \mathrm{deg}_\omega(v) = 2\}.
\]

The graphical representation below is inspired by the Edwards--Sokal coupling~\cite{EdwSok88} and the work of Chayes and Machta~\cite{ChaMac97}. 
In a particular case of~$n=2$, this representation was introduced by the second author and Spinka; it is described in~\cite{GlaLam25}. 

\begin{definition}[Blocking vertices]\label{def:xi}
	Let~$\omega$ be a loop configuration on~$\hexlattice$.
	A site percolation~$\xi$ on~$\upvert(\hexlattice)$ is defined as follows: 
	\begin{itemize}
		\item each~$v$ of degree~$0$ in~$\omega$ is open in~$\xi$ with probability~$1-x^2$ and otherwise closed, independently of the others;
		\item for each loop~$\ell$ of~$\omega$, all vertices of~$\ell$ are simultaneously open in~$\xi$  with probability~$(n-1)/n$ and otherwise closed, independently of other loops.
	\end{itemize}
We emphasize that, if $\omega$ contains a bi-infinite path, all vertices on that path are closed in $\xi$. Define~$\mu^\omega$ as the law of~$\xi$ given~$\omega$.
	We identify~$\xi$ with the set~$\{\xi=1\}$ that we call {\em blocking vertices}.  
\end{definition}

The next proposition is the key to our proof of Theorem~\ref{thm:LoopOn} and follows from Corollory~\ref{cor:divide-color}.

\begin{proposition}[Blocking vertices do not percolate]
\label{prop:no-perc-xi}
	Let~$n\in [1,2]$ and~$x\in [1/\sqrt{2},1]$.
	Consider any translation-invariant Gibbs measure~$\P$ for the loop~\O{n} model with edge-weight~$x$.
	Sample a loop configuration~$\omega$ from~$\P$ and then~$\xi$ from~$\omega$.
	Then, $\xi$ exhibits no infinite connected component, almost surely.
\end{proposition}

\begin{proof}
	The triangular lattice~$\upvert(\hexlattice)$ is transitive and amenable.
	The loops of~$\omega$ induce its finitary invariant partition: two vertices of~$\upvert(\hexlattice)$ are equivalent if and only if they belong to the same loop of~$\omega$.
	Also, $\xi$ is stochastically dominated by Bernoulli percolation on loops and vertices at parameter~$1/2$, since~$(n-1)/n\leq 1/2$ and~$1-x^2\leq 1/2$.
	Then, by Corollary~\ref{cor:divide-color}, $\xi$ does not percolate, almost surely.
\end{proof}

Define~$\T=(\vertices\T, \edges\T)$ as the triangular lattice dual to~$\hexlattice$: its vertex-set is given by $\{k+\ell
e^{i\pi/3}:k,\ell\in\Z\}\subset\mathbb C$, and the edges connect the nearest neighbors.
We will identify vertices of~$\T$ with faces of~$\hexlattice$.
For each~$\sigma \colon \vertices\T \to \pm 1$, define~$\mathrm{DW}(\sigma)$ as the set of domain walls of~$\sigma$: 
a spanning subgraph of~$\hexlattice$ defined by edges that separate faces of~$\hexlattice$ with opposite spins.
Note that~$\mathrm{DW}(\sigma)$ is a loop configuration on~$\hexlattice$.
Moreover, each loop configuration on~$\hexlattice$ has exactly two preimages under this mapping, and they differ by a global spin flip.
 
For~$\xi\in \{0,1\}^{\upvert(\hexlattice)}$, define~$\Delta(\xi)$ a bond percolation on~$\T$ given by the edges belonging to the (upward oriented triangular) faces of~$\T$ corresponding to the vertices in~$\xi$.
Remarkably, two vertices are connected in~$\xi$ if and only if any two faces containing them are connected in~$\Delta(\xi)$.
In particular, $\xi$ has an infinite connected component if and only if~$\Delta(\xi)$ does.

\begin{proof}[Proof of Theorem~\ref{thm:LoopOn}]

	Sample~$\omega$ with respect to~$\P$, and then a configuration~$\xi$ of blocking vertices with respect to~$\mu^\omega$.
	Given~$\omega$ and~$\xi$, define~$\omega^{\mathrm{free}}$ as the part of~$\omega$ not covered by~$\xi$, and take~$\sigma^{\mathrm{bc}}\in \mathrm{DW}^{-1}(\omega^{\mathrm{free}})$ such that~$\sigma^{\mathrm{bc}}(0)=1$.
	
	By Proposition~\ref{prop:no-perc-xi}, $\xi$ does not percolate~$\mu^\omega$-a.s. for~$\P$-a.e. $\omega$.
	Then, each connected component of~$\Delta(\xi)$ is also finite a.s.
	Now sample a random~$\sigma\in\{\pm 1\}^\T$ by assigning~$1$ or~$-1$ to each component of~$\Delta(\xi)$ independently with probability~$1/2$.
	Finally, define~$\tilde{\sigma} \in\{\pm 1\}^\T$ that coincides with~$\sigma$ on the components of~$\Delta(\xi)$ contained completely inside~$\ballloop_r$ and with~$\sigma^{\mathrm{bc}}$ on all other components of~$\Delta(\xi)$.
	
	Assume that~$\omega$ contains finitely many loops surrounding the origin with positive probability.
	Then, for some~$q\in\N$ and~$\varepsilon>0$, 
	\begin{equation}\label{eq:no-loops-in-omega}
		\P(\omega \text{ contains a loop surrounding } \ballloop_q) < 1-\varepsilon.
	\end{equation}
	By Corollary~\ref{cor:divide-color} applied to~$\sigma$ conditioned on~$\xi$, we get that~$\sigma$ does not contain infinite clusters of neither pluses nor minuses, almost surely for any~$\xi$.
	Averaging over~$\omega$ and~$\xi$, we get that there exists~$r\in \N$ such that
	\[
		\P\left(\mathrm{DW}(\sigma) \text{ contains at least two loops in } \ballloop_r \text{ surrounding } \ballloop_q \right) \geq 1-\varepsilon.
	\]
	Take any such~$\sigma$ and denote the outermost such loop by $\mathrm{loop}_\mathrm{out}$, and set $\mathrm{loop}_\mathrm{in}$ to be another one. 
	Then, for {\em any} consistent choice of $\xi$,  any cluster of~$\Delta(\xi)$ intersecting~$\partial\ballloop_r$ {\em does not} intersect the interior of~$\mathrm{loop}_\mathrm{out}$. In particular, $\mathrm{loop}_\mathrm{in} \in \mathrm{DW}(\tilde{\sigma})$.
	In conclusion,
	\begin{equation}\label{eq:SigmatildeCoupling}
		\P\left(\mathrm{DW}(\tilde{\sigma}) \text{ contains a loop in } \ballloop_r \text{ surrounding } \ballloop_q \right) \geq 1-\varepsilon.
	\end{equation}
	This will contradict~\eqref{eq:no-loops-in-omega} once we have shown that~$\mathrm{DW}(\tilde{\sigma})$ and~$\omega^{\mathrm{free}}\subset\omega$ have the same law.

	From now on, we condition on a realization of~$\omega$ outside of~$\ballloop_r$.
	Note that~$\mathrm{DW}(\tilde{\sigma})$ and~$\omega^{\mathrm{free}}$ deterministically coincide outside of~$\ballloop_r$.
	It will be convenient to consider~$R>r$ such that all finite components of~$\omega\cap \ballloop_r^c$ that intersect~$\ballloop_r$ are contained in~$\ballloop_R$.
	Define~$\omega^{\mathrm{free}}_r$ (resp.~$\omega^{\mathrm{block}}_r$) as spanning subgraphs of~$\ballloop_R$ consisting of the connected components of~$\omega^{\mathrm{free}}\cap \ballloop_R$ (resp.~$(\omega\setminus\omega^{\mathrm{free}})\cap \ballloop_R$) that intersect~$\ballloop_r$.
	By the definitions of~$\xi$ and~$R$, $\omega^{\mathrm{block}}_r$ consists of loops, while~$\omega^{\mathrm{free}}_r$ might {\em a priori} contain also paths starting and ending on~$\partial \ballloop_R$.
	
	Let $\bar{\omega}_r := \omega^{\mathrm{block}}_r \sqcup \omega^{\mathrm{free}}_r$. 
	 Note that~$\bar{\omega}_r$ consists of loops and paths starting and ending at~$\partial\ballloop_R$. Moreover, $\bar{\omega}_r$ has the same restriction to $\ballloop_R \setminus \ballloop_r$ as the union of the components of $\omega$ which intersect $\ballloop_r$.
	Define~$F_{r,R}^{\omega}$ as the set of all spanning subgraphs of~$\ballloop_R$ that satisfy these properties.
	By the DLR condition, the restriction of $\bar{\omega}_r$ to~$\ballloop_r$ is distributed as ~$\muLoop_{\ballloop_r,n,x}^{\omega}$.
	Since each loop is in~$\omega^{\mathrm{block}}_r$ with probability~$(n-1)/n$, independently of the others, we get
	\begin{align*}
		(\omega^{\mathrm{block}}_r,\omega^{\mathrm{free}}_r)
		&\propto n^{\ell(\omega^{\mathrm{block}}_r) + \ell(\omega^{\mathrm{free}}_r)}
		\cdot x^{|\bar{\omega}_r|}
		\cdot 
		\left(\tfrac{n-1}{n}\right)^{\ell(\omega^{\mathrm{block}}_r)}
		\cdot 
		\left(\tfrac{1}{n}\right)^{\ell(\omega^{\mathrm{free}}_r)} \cdot		\1_{ \bar{\omega}_r \in F_{r,R}^{\omega} }
		\\
		&\propto
		(n-1)^{\ell(\omega^{\mathrm{block}}_r)} 
		\left(\tfrac{1}{x^2}\right)^{|\upvert(\ballloop_r)\setminus \; \bar{\omega}_r |}
		\cdot
		\1_{ \bar{\omega}_r \in F_{r,R}^{\omega} },
	\end{align*}
where we use that each loop on~$\hexlattice$ contains twice as many edges as vertices in~$\upvert(\hexlattice)$; this identity holds also for paths in~$\omega^{\mathrm{free}}_r$, up to constant boundary corrections.

	Denote by~$\xi_r$ the restriction of~$\xi$ to the loops that intersect~$\ballloop_r$ and to vertices in~$\upvert(\ballloop_r)$.
	By definition, $\upvert(\omega^{\mathrm{block}}_r)\subseteq \xi_r$ and~$\upvert(\omega^{\mathrm{free}}_r)\cap \xi_r = \emptyset$, and each vertex in~$\upvert(\ballloop_r)\setminus \upvert(\omega^{\mathrm{block}}_r\sqcup\omega^{\mathrm{free}}_r)$ belongs to~$\xi$ with probability~$1-x^2$, independently of the others.
	Thus,
	\begin{align*}
		(\omega^{\mathrm{block}}_r,\omega^{\mathrm{free}}_r,\xi_r)
		\propto
		(n-1)^{\ell(\omega^{\mathrm{block}}_r)} 
		\left(\tfrac{1}{x^2}-1\right)^{|\xi_r\setminus \upvert(\omega^{\mathrm{block}}_r)|}
		\hspace{-1mm}
		\cdot 
		\1_{\upvert(\omega^{\mathrm{block}}_r)\subseteq \xi_r}
		\hspace{-0.5mm}
		\cdot 
		\hspace{-0.5mm}
		\1_{\upvert(\omega^{\mathrm{free}}_r)\cap \xi_r = \emptyset, \, \bar{\omega}_r \in F_{r,R}^{\omega} }.
	\end{align*}
	This distribution depends on~$\omega^{\mathrm{free}}_r$ only via the last indicator.
	Thus, given any pair~$(\omega^{\mathrm{block}}_r,\xi_r)$ of non-zero probability, the conditional distribution of~$\omega^{\mathrm{free}}_r$ is uniform on the set of spanning subgraphs of~$\ballloop_R$ in which vertices of~$\xi$ have degree zero and which satisfy $\bar{\omega}_r \in F_{r,R}^{\omega}$.
	
	Now define~$\mathrm{DW}_r(\tilde{\sigma})$ as the union of the connected components of~$\mathrm{DW}(\tilde{\sigma})$ in~$\ballloop_R$ that intersect~$\ballloop_r$.
	By the definition of~$\tilde{\sigma}$, the distribution of~$\mathrm{DW}_r(\tilde{\sigma})$ is uniform over some set of spanning subgraphs of~$\ballloop_R$.
	Since $\mathrm{DW}(\tilde{\sigma})$ and~$\omega^{\mathrm{free}}$ coincide on~$\ballloop_r^c$, we immediately verify that $\omega^{\mathrm{block}}_r \sqcup \mathrm{DW}_r(\tilde{\sigma}) \in F_{r,R}^{\omega}$.
	Also, $\tilde{\sigma}$ is constant around vertices in~$\xi$, whence loops and paths of~$\mathrm{DW}_r(\tilde{\sigma})$ avoid~$\xi$.
	Finally, the support of~$\omega^{\mathrm{free}}_r$ is inside that of~$\mathrm{DW}_R(\tilde{\sigma})$ because~$\sigma^\mathrm{bc}$ is a possible realization of~$\tilde{\sigma}$ and~$\mathrm{DW}_r(\sigma^\mathrm{bc}) = \omega^{\mathrm{free}}_r$.
	
	This implies that~$\mathrm{DW}(\tilde{\sigma})$ and~$\omega^{\mathrm{free}}\subset\omega$ have the same law, whence~\eqref{eq:no-loops-in-omega} and~\eqref{eq:SigmatildeCoupling} are in contradiction.
	As a consequence, $\omega$ contains infinitely many loops around every face, almost surely.
	When~$n\geq 1, x\leq 1/\sqrt{n}$, this readily implies ~\eqref{eq:RSW} by~\cite[Theorem~$2$]{DumGlaPel21}.
\end{proof}

\bibliographystyle{amsalpha}
\bibliography{biblicomplete.bib}

\end{document}